\documentclass[a4paper]{amsart}

\usepackage{latexsym, amssymb,amsmath, amsthm,amsfonts}

\newcommand{\F}{\mathbb{F}}

\newcommand{\Z}{\mathbb{Z}}
\newcommand{\kk}{\Bbbk}

\newcommand{\dd}{\mathrm{d}}

\def\GL{\operatorname{GL}}
\def\SL{\operatorname{SL}}
\def\PSL{\operatorname{PSL}}

\def\LT{\operatorname{LT}}
\def\res{\operatorname{res}}

\def\chr{\operatorname{char}}
\def\Quot{\operatorname{Quot}}

\def\Z{\mathbb{Z}}
\def\N{\mathbb{N}}

\def\rk{\operatorname{rank}}

\def\codim{\operatorname{codim}}

\def\tr{\operatorname{trace}}
\def\Jac{\operatorname{Jac}}

\newtheorem{Lemma}{Lemma}[section]
\newtheorem{Theorem}[Lemma]{Theorem}

\newtheorem{Corollary}[Lemma]{Corollary}
\newtheorem{cor}[Lemma]{Corollary}
\newtheorem{prop}[Lemma]{Proposition}

\newtheorem*{Corollary of Conjecture}{Corollary of Conjecture}

\theoremstyle{definition}

\theoremstyle{remark}
  \newtheorem{rem}[Lemma]{Remark}

\newtheoremstyle{Acknowledgments}
  {}
    {}
     {}
     {}
    {\bfseries}
    {}
     {.5em}
     {\thmname{#1}\thmnumber{ }\thmnote{ (#3)}}
\theoremstyle{Acknowledgments}
\newtheorem{ack}{Acknowledgments.}

\title{Conjugation Differential invariants of $\SL_2(\F_q)$ on trace-free matrices}

\author{Jonathan Elmer}
\address{Middlesex University\\
The Burroughs, Hendon, London\\
NW4 4BT UK}
\email{j.elmer@mdx.ac.uk}

\begin{document}

\maketitle

\begin{abstract}
Let $q = p^k$ be  prime power, let $F = \F_q$ and let $V$ denote the vector space of $2 \times 2$ matrices with coefficients in $F$ and trace zero. Let $G = \SL_2(F)$. Then $G$ acts on $V$ via conjugation. Let $\Omega =(S(V^*) \otimes \Lambda(V^*))$ be the algebra of differential forms on $V$. A minimal generating set for $\Omega^G$ in the case $q=3$ was constructed in \cite{ElmerMeyer}. In this article we construct a minimal generating set for $\Omega^G$ for all $q$. 
\end{abstract}

\section{Introduction}
\subsection{Differential invariants}\label{sec:diffinv}

Let $\kk$ be a field, $G$ a finite group and $V$ a finite-dimensional left $\kk G$-module. Let $v_1,v_2, \ldots, v_n$ be a basis of $V$ and let $x_1,x_2, \ldots, x_n$ be the corresponding dual basis of $V^*$. The action of $G$ on $V$ induces a left action of $G$ on $V^*$ defined by
\[(g \cdot f)(v) = f(g^{-1} \cdot v)\]
for $f \in V^*, g \in G$ and $v \in V$, and we may extend this, by algebra automorphisms, to an action on the $\kk$-algebra $R:=S(V^*)$. Setting $\deg(x_i)=1$ for each $i$ endows $S(V^*)$ with a grading which is respected by the $G$-action. 
The fixed points $R^G$ of this action form a graded subalgebra called the {\it algebra of invariants.}

Now let $W$ be a left $\kk G$-module with basis $w_1, w_2, \ldots, w_m$. Let $R \otimes W$ denote the free left $R$-module spanned by this basis. $G$ acts on $R \otimes W$ via the formula
\[g \cdot (f \otimes w)  = (g \cdot f) \otimes (g \cdot w)\] for $f \in R$, $w \in W$, and $g \in G$, and one may now show that the fixed points $(R \otimes W)^G$ of this action form an $R^G$-module. This is called the {\it module of $W$-covariants.} Since $R$ is finitely generated over $R^G$ in general, and $R^G$ is Noetherian, it follows that $(R \otimes W)^G \subseteq (R \otimes W)$ is a finitely generated $R^G$-module.

Consider the special case $\Omega^i:= R \otimes (\Lambda^i(V^*))$. This is called the module of differential $i$-forms on $V$. To explain the terminology, write $\dd x_1,\dd x_2, \ldots, \dd x_n$ for the $R$-basis of $\Omega^1$, and observe that the map $\dd: R \rightarrow \Omega^1$ defined by
\begin{equation}\label{differential}
\dd f: = \sum_{j=1}^n \frac{\partial f}{\partial x_j} \dd x_j
\end{equation}
is $G$-equivariant. The exterior product $\wedge: \Lambda^i(V^*) \times \Lambda^{j}(V^*) \rightarrow \Lambda^{i+j}(V^*)$ can be extended $R$-linearly to an exterior product on $\Omega:= \bigoplus_{i=0}^n \Omega^i = S(V^*) \otimes \Lambda(V^*)$ (note that $\Lambda^i(V^*) = 0$ for $i>n)$ which makes $\Omega$ into a non-commutative $R$-algebra. This algebra has a natural bigrading which is respected by the $G$-action. The fixed points $\Omega^G$ form a bigraded subalgebra of $\Omega$ called the algebra of {\it differential invariants.} 

\begin{rem}
Another way of viewing $\Omega$ is as follows: write $y_i:= \dd x_i$ and define a $\N$-grading on $\Omega$ by setting $\deg(x_i)=2$ and $\deg(y_i)=1$ for all $i=1, \ldots, n$. Then $\Omega$ is a commutative-graded $\kk$-algebra, i.e. for each homogeneous pair $f_1, f_2 \in \Omega$ we have
\[f_1f_2 = (-1)^{\deg(f_1)\deg(f_2)}f_2f_1.\] Although this is the natural setting for applications to cohomology theory, we will not use this convention in this article.
\end{rem}

\subsection{Background}

Rings of invariants and modules of covariants are widely studied in the non-modular case, i.e. when  the characteristic of  $\kk$ does not divide $|G|$. Perhaps the most important result is that of Chevalley and Shepherd-Todd \cite{Chevalley}, which states that $R^G$ is a polynomial ring if and only if $(R \otimes W)^G$ is a free $R^G$-algebra for each $W$, if and only if $G$ acts on $V$ as a reflection group. Another important result, due to Eagon and Hochster \cite{EagonHochster} states that $(R \otimes W)^G$ is always a Cohen-Macaulay $R^G$-module, and in particular, $R^G$ is always a Cohen-Macaulay ring.

Rings of differential invariants are also widely studied in the non-modular case. One reason for this is as follows: let $G$ be a finite group of order divisible by $p$ and let $P$ denote a Sylow-$p$-subgroup. Suppose $P$ is elementary abelian of rank $n$ and $P$ is normal in $G$. Then 
\[H(G,\F_p) \cong \Omega^{G/P},\]
see \cite[Section~5.1]{Benson}.
In the above, $V = \langle v_1,v_2, \ldots, v_n \rangle$ where $\{v_1, v_2, \ldots, v_n\}$ is a basis for $P$ as an $\F_p$-vector space and $G/P$ acts by conjugation.
The most important result in this area is Solomon's theorem \cite[Theorem~7.3.1]{Benson}, which states that if $G$ acts on $V$ as a reflection group (so $R^G = \kk[f_1, \ldots, f_n]$ for some $f_1, \ldots, f_n \in R$), then
\[\Omega^G = \kk[f_1, \ldots, f_n] \otimes \Lambda (\dd f_1, \ldots, \dd f_n).\]
Note that in addition it is a clear consequence of Eagon and Hochster's theorem that $\Omega^G$ is a Cohen-Macaulay ring, since $\Lambda^*(V^*)$ is a finite-dimensional vector space.

In the modular case, rather less is known in general. For algebras of invariants, it is known by a theorem of Serre that if $R^G$ is a polynomial ring then $G$ must act by reflections \cite[Theorem~3.7.8]{DerksenKemper}, but the converse is not true in general \cite{KemperOnCM}. In addition, $R^G$ is not Cohen-Macaulay in general, but rather it is known that if $R^G$ is a Cohen-Macaulay ring and $G$ is a $p$-group, then $G$ must act by bireflections \cite{KemperOnCM}. In the other direction, Ellingsrud and Skjelbred \cite{Ellingsrud} proved that if $\codim_V(V^G) \leq 2$ then $R^G$ is Cohen-Macaulay. This was extended by Broer and Chuai \cite{BroerChuaiRelative}, who proved that under that same hypothesis $(R \otimes W)^G$ is a Cohen-Macaulay $R^G$-module for every choice of $W$.

In the modular case, Hartmann \cite{Hartmann} proved that the analogue of Solomon's theorem holds if and only if $R^G$ is a polynomial algebra and, in addition, $G$ contains no transvections. Hartmann and Shepler \cite{HartmannShepler} later studied the case where $G$ contains transvections but $R^G$ is still polynomial, and found that in many cases $\Omega^G$ is a free $R^G$-module and has a ``twisted'' free exterior algebra structure. 

The author \cite{ElmerCMCov} computed minimal generating sets for $\Omega^G$ in the case where $G$ is cyclic of order $p$ and $V$ is indecomposable of dimension at most 3. In the dimension 2 case, $R^G$ is polynomial and $\Omega^G$ is a free twisted exterior algebra. In the dimension 3 case, $R^G$ is Cohen-Macaulay and $\Omega^G$ is a Cohen-Macaulay $R^G$ module as expected, but the structure of $\Omega^G$ as a $\kk$-algebra is more complicated.

Modular differential invariants can be useful in cohomology calculations. For example, let $F = \F_p$ where $p>2$ and let $V$ denote the vector space of $2 \times 2$ matrices with coefficients in $F$ and trace zero. Then $G:=\SL_2(F)$ acts on $V$ via conjugation.  Recent work of Meyer \cite{Meyer3} has shown that, for all $n \geq 2$, $H(\SL_2(\Z/p^n), F)$ is closely connected to $\Omega^{\SL_2(F)}$. In particular if $p \geq 3$ we have we have, for all $n \geq 2$, $H(K_n, F) \cong \Omega$ as $FG$-modules and if $p=3$ then
\[H^*(\SL_2(\Z/p^n), F) = \{x \in H^*(P_n, F): \res^{P_n}_{K_n} \in \Omega^G\}\]
where $P_n$ is a Sylow-$p$-subgroup of $\SL_2(\Z/p^n)$ and $K_n$ is the kernel of the map $\SL_2(\Z/p^n) \rightarrow \SL_2(\Z/p)$ induced by reduction modulo $p$. 

\subsection{Notation and results}\label{sec:not}

From this point onwards we fix a prime $p$, a prime power $q=p^k$, and set $F:= \F_q$. $F^*$ denotes the non-zero elements of $F$. We let $V$ denote the vector space of $2 \times 2$ traceless matrices over $F$. $SL_2(F)$ acts by conjugation on $V$ with kernel  $K:= \{\pm I \},$ where $I$ is the $2 \times 2$ identity matrix. We set
\[G:= \SL_2(F)/K = \PSL_2(F).\]

The goal of the present article is to construct a minimal generating set for $\Omega^G$. Along the way it is necessary to describe the polynomial invariants $R^G$:

\begin{Theorem}\label{invars}\
\begin{enumerate}
\item[(a)] Suppose $q$ is even. Then $R^G$ is a polynomial ring
\[R^G = F[f_1,f_2,f_3].\]
\item[(b)] Suppose $q$ is odd. Then $R^G$ is a hypersurface. More precisely
\[\mathcal{A}:= F[f_1,f_2,f_3]\]
is a homogeneous system of parameters for $R^G$, and $R^G$ is freely generated over $\mathcal{A}$ by $1$ and $f_4$.
\end{enumerate}
In the above $f_1$ is the determinant of a generic tracefree matrix, $f_2$ its first Steenrod square, $f_3$ is an invariant of degree $\frac12 q(q-1)$ which we define in Section \ref{sec:inv}, and $f_4 = \Jac(f_1,f_2,f_3)$.
\end{Theorem}
This result was originally proved by Anghel \cite{Anghel}, and partially simplified by Smith \cite{Smith2x2} in the odd case. A more thorough treatment was given by Maithani \cite{Maithani} but their description of $f_3$ for odd $p$ is not convenient for our purpose. We give an alternative description of $f_3$ for odd $p$ and show that our results are equivalent in Section \ref{sec:odd}.

Our main result for even $q$ is as follows:

\begin{Theorem}\label{even}\
\begin{itemize}
\item[(a)] Suppose $q>2$ is even. Then $\Omega^G$ is  free $R^G$ module generated by
\begin{align*}
&\{1, \\
&\dd f_1, \dd f_2, \omega, \\
&f_2^{-1}\dd f_1 \wedge \dd f_2, f_2^{-1}\dd f_1 \wedge \omega,f_2^{-1}\dd f_2 \wedge \omega\\
&\dd x_1 \wedge \dd x_2 \wedge \dd x_3\}.
\end{align*} where $\omega \in (\Omega^1)^G$ has polynomial degree $q+1$ and is defined in Section \ref{sec:even}.

\item[(b)] Suppose $q=2$. Then $\Omega^G$ is  free $R^G$ module generated by
\begin{align*}
&\{1, \\
&\dd f_1, \dd f_2, \dd f_3, \\
&f_2^{-1}\dd f_1 \wedge \dd f_2, \dd f_1 \wedge \dd f_3, \dd f_2 \wedge \dd f_3\\
&\dd x_1 \wedge \dd x_2 \wedge \dd x_3\}.
\end{align*} 
\end{itemize}
\end{Theorem}

\begin{Corollary}\label{evenmin}\
\begin{itemize}
\item[(a)] Suppose $q>2$ is even. Then a minimal generating set for $\Omega^G$ is  given by
\begin{align*}
&\{f_1, f_2, f_3, \\
&\dd f_1, \dd f_2, \omega, \\
&f_2^{-1}\dd f_1 \wedge \dd f_2, f_2^{-1}\dd f_1 \wedge \omega,f_2^{-1}\dd f_2 \wedge \omega\\
&\dd x_1 \wedge \dd x_2 \wedge \dd x_3\}.
\end{align*}

\item[(b)] Suppose $q=2$. Then a minimal generating set for $\Omega^G$ is  given by
\begin{align*}
&\{f_1, f_2, f_3, \\&\dd f_1, \dd f_2, \dd f_3, \\
&f_2^{-1}\dd f_1 \wedge \dd f_2\}.
\end{align*} 
\end{itemize}
\end{Corollary}

We prove these results in Section \ref{sec:even}. As $G$ is a reflection group, our methods depend heavily on the work of Hartmann, Shepler, and Hanson \cite{HartmannShepler, HansonShepler} on differential invariants of reflection groups. We outline the results we need in Section \ref{sec:refs}.

In order to state our main result for odd $q$, recall that in this case we have an isomorphism $V \cong V^*$ (this is made explicit in Section \ref{sec:odd}). This induces an isomorphism $\phi:(\Omega^2)^G \rightarrow (\Omega^1)^G$.

\begin{Theorem}\label{odd}
Suppose $q$ is odd. Then $\Omega^G$ is a Cohen-Macaulay $R^G$-module, freely generated over $\mathcal{A}$ by
\begin{align*} S =\{&1,f_4,\\ &\dd f_1, \dd f_2, \dd f_3, \phi(\dd f_2 \wedge \dd f_3), \phi(\dd f_3 \wedge \dd f_1), \phi(\dd f_1 \wedge \dd f_2),\\
&\phi^{-1}(\dd f_1), \phi^{-1}(\dd f_2), \phi^{-1}(\dd f_3), \dd f_2 \wedge \dd f_3, \dd f_3 \wedge \dd f_1, \dd f_1 \wedge \dd f_2,\\
& \dd x_1 \wedge \dd x_2 \wedge \dd x_3, f_4 \dd x_1 \wedge \dd x_2 \wedge \dd x_3 \}.\end{align*}
\end{Theorem}

\begin{Corollary}\label{oddmin}
Suppose $q$ is odd. A minimal generating set for $\Omega^G$ is given by
\begin{align*}\{&f_1,f_2,f_3,f_4,\\ &\dd f_1, \dd f_2, \dd f_3, \phi(\dd f_2 \wedge \dd f_3), \phi(\dd f_3 \wedge \dd f_1), \phi(\dd f_1 \wedge \dd f_2),\\
&\phi^{-1}(\dd f_1), \phi^{-1}(\dd f_2), \phi^{-1}(\dd f_3), \\
& \dd x_1 \wedge \dd x_2 \wedge \dd x_3.\}.\end{align*}
\end{Corollary}

We prove these results for all $q \neq 3$ in Section \ref{sec:odd}. The corresponding results for $q=3$ were given in \cite{ElmerMeyer}; we verify that the statements are equivalent in Section \ref{sec:three}.

\begin{ack} The author wishes to thank Anja Meyer for sharing an application of the problem in the case $p=q$ odd to cohomology theory, and Akiyoshi Sannai for pointing out a couple of mistakes in an earlier version of this paper.
\end{ack}

\section{Covariants of Reflection Groups}\label{sec:refs}

In this section we will use the notation of Section \ref{sec:diffinv}: so $\kk$ is a field of arbitrary characteristic, $V$ a vector space over $\kk$ with dimension $n$ and $G$ a finite group acting on $V$ on the left. Let $\rho: G \rightarrow \GL_n(\kk)$ be the homomorphism associated with the representation.

Recall that $g \in G$ is said to be a {\it reflection} if $g  \neq 1$ and $g$ fixes a subspace $U$ of $V$ with codimension 1. We say $g$ is reflection about $U$ and that $U$ is the {\it reflecting hyperplane} of $g$. If $|g|$ is coprime to $\chr(\kk)$ then (possibly after extending $\kk$) one can choose a basis of $V$ with respect to which $g$ is diagonal, with all entries 1 except the last which will be a nontrivial $|g|$th root of 1 in $\kk$. We call these diagonalisable reflections. If $|G|$ is coprime to $\chr(\kk)$, all reflections are of this kind. 

If $p=\chr(\kk)| |G|$ then there may exist reflections of order $p$, which we call {\it transvections}. If in addition $\rho(G) \subseteq \SL_n(\kk)$ then all reflections are transvections.

Let $g \in G$ be a reflection about $U$. Write $U = \ker(l_U)$ for some $l_U \in V^*$, which is unique up to scalar. Then there exists some vector $v_g \in V$ (again, unique up to a scalar which depends on the choice of $l_U$) such that, for all $w \in V$ we have
\[g(w) = w + l_U(w)v_g.\]
$v_g$ is called the {\it root vector} associated to $g$, and it is easily seen that $g$ is a transvection if and only if $v_g \in U$. In the other direction, given a reflecting hyperplane $U \leq V$, one may consider the {\it root space} $R_U$ spanned by the root vectors of all reflections about $U$, and $T_U = R_U \cap U$, the {\it transvection root space} spanned by the root vectors of all transvections about $U$. We write $b_U = \dim(T_U)$. If $b_U=n-1$, so that $U=T_U$, we say the transvection root space of $U$ is {\it maximal}.

Denote by $G_U$ the stabiliser of $U$ in $G$. The transvections about $U$ along with the identity form a normal subgroup $K_U$ of $G_U$; we write $e_U = |G_U:K_U|$. Note that $K_U = \ker(\det|_{G_U})$; in particular if $\rho(G) \subseteq \SL_n(\kk)$ then $e_U=1$.

Let $\mathcal{U}$ denote the set of reflecting hyperplanes in $V$. Associated to $\mathcal{U}$ one has the following polynomials:
\[Q = \prod_{U \in \mathcal{U}} l_U; \qquad Q_{\det}= \prod_{U \in \mathcal{U}} l_U^{e_U-1}; \qquad \tilde{Q} =\prod_{U \in \mathcal{U}} l_U^{b_Ue_U}.\]

Hanson and Shepler considered covariant modules of the form $(R \otimes V^*)^G$ and $(R \otimes V)^G$, which they callled invariant differential 1-forms and invariant derivations respectively. The following is a version of Theorem 3.2(a) in \cite{HansonShepler}:

\begin{prop}\label{1formcriterion}
Let $\omega_1, \omega_2, \ldots, \omega_n \in (R \otimes V^*)^G$. The following are equivalent:
\begin{enumerate}
\item These elements generate $(R \otimes V^*)^G$ freely as an $R^G$ module;
\item $\omega_1 \wedge \omega_2 \ldots \wedge \omega_n = Q_{\det} \tilde{Q} \dd x_1 \wedge \dd x_2 \ldots \wedge \dd x_n;$
\item These elements are linearly independent over $\Quot(R)$ and $$\sum_{i=1}^n \deg(\omega_i) = \sum_{U \in \mathcal{U}}(b_Ue_U+e_U-1).$$
\end{enumerate}
\end{prop}

We will also need the following version of Theorem 3.2(b) in \cite{HansonShepler} providing a  similar result for $(R \otimes V)^G$: in order to formulate it properly, we let $\delta_1, \ldots, \delta_n$ be a basis of $R_0 \otimes V$ with the action the same as the action on $v_1, \ldots, v_n$ and view $R \otimes V$ as a direct summand of $R \otimes \Lambda(V)$ with the wedging operation ended $R$-linearly. 

\begin{prop}\label{derivcriterion}
Let $\psi_1, \psi_2, \ldots, \psi_n \in (R \otimes V)^G$. The following are equivalent:
\begin{enumerate}
\item These elements generate $(R \otimes V)^G$ freely as an $R^G$ module;
\item $\psi_1 \wedge \psi_2 \ldots \wedge \psi_n = Q \delta_1 \wedge \delta_2 \ldots \wedge \delta_n;$
\item These elements are linearly independent over $\Quot(R)$ and $$\sum_{i=1}^n \deg(\psi_i) = |\mathcal{U}|.$$
\end{enumerate}
\end{prop}

\section{The invariant algebra}\label{sec:inv}

For the rest of this article we adopt the notation of Section \ref{sec:not}. We choose a basis
\[v_1 = \begin{pmatrix} 0 & -1 \\ 0 & 0 \end{pmatrix}, v_2 =  \begin{pmatrix} 1 & 0 \\ 0 & -1 \end{pmatrix}, v_3 =  \begin{pmatrix} 0 & 0 \\ 1 & 0 \end{pmatrix} \]
for $V$. Let $x_1,x_2,x_3$ be the corresponding dual basis for $V^*$, which will generate $R$ as a commutative algebra.
We will use a grevlex order on $R$ with $x_1>x_2>x_3$. We will use the terms $x_i$-degree to mean the usual degree of a polynomial in $R$ viewed as a polynomial in $x_i$ alone. 

We define an action of $\SL_2(F)$ on $V$ by setting
\[g \cdot v = gvg^{-1}\]
for $g \in \SL_2(F)$ and $v \in V$. This action has kernel $K$ which is nontrivial if $q$ is odd, and we set $G = \SL_2(F)/K$. Elements of $G$ will be written as matrices with square brackets. Note that if $q$ is even then $G = \SL_2(F)$, and in particular the trace of an element of $G$ is well-defined.

The order of $G$ is
\begin{equation}\label{evenorder} q(q^2-1)
\end{equation}
if $q$ is even and
\begin{equation}\label{oddorder} \frac12 q(q^2-1)
\end{equation}
if $q$ is odd.

We label certain elements of $G$: let $e \in F$ and let 
\[\tau_e:= \begin{bmatrix} 1 & e\\ 0 & 1 \end{bmatrix}.\]
Then $\{\tau_e: e \in F\}$ is a Sylow-$p$-subgroup $P \cong C^k_p$ of $G$. Let
\[\sigma:= \begin{bmatrix} 0 & 1 \\ -1 & 0 \end{bmatrix}.\] We have $\langle P, \sigma \rangle = G$.
Let $\rho: G \rightarrow \GL_3(F)$ be the representation afforded by the action of $G$ on $V$. Then we compute
\[\rho(\tau_e) = \begin{pmatrix} 1 & 2e & e^2\\ 0 & 1 & e\\ 0 & 0 & 1 \end{pmatrix},\]
and
\[\rho(\sigma) = \begin{pmatrix} 0 & 0 & 1\\ 0 & -1 & 0 \\ 1 & 0 & 0 \end{pmatrix}\]
Notice that:
\begin{Lemma}\label{insl3} 
$\rho(G) \subseteq \SL_3(F)$.
\end{Lemma}

More generally, for the element $g(a,b,c,d):= \begin{bmatrix} a&b\\c&d \end{bmatrix} \in G$ we have
\begin{equation}\label{rho}\rho(g) = \begin{pmatrix} a^2 & 2ab & b^2 \\ ac & ad+bc & bd \\ c^2 & 2cd & d^2   \end{pmatrix}.\end{equation}

\begin{Lemma}\label{norefs}\
\begin{enumerate}
\item[(a)] Suppose $q$ is even. Then $g \in G$ is a reflection if and only if $g$ has trace zero.
\item[(b)] Suppose $q$ is odd. Then no element of $G$ acts as a reflection on $V$.
\end{enumerate}
\end{Lemma}

\begin{proof} Recall that $g \in G$ is a reflection if and only if $\rk(\rho(g)-I)=1$, equivalently if and only if $\dim \ker(\rho(g)-I)=2$. Let $\hat{g} \in \SL_2$ be a representative of $g$. Identifying $V$ and $F^3$, we see that $A \in \ker(\rho(g)-I)$ if and only if $\hat{g}A=A\hat{g}$, independently of the chosen representative. 
As these dimensions are not affected by extending the field, and fixed, we can pass to an algebraic closure $\overline{F}$ of $F$. As they are fixed under conjugation in $\GL_2(\overline{F})$, we may assume $g \in G$ is represented by a matrix in $\SL_2(F)$ in Jordan form. If this matrix is semi-simple, then its eigenvalues are $\lambda$ and $\lambda^{-1}$ for some $\lambda \in \overline{F}^*$, and the eigenvalues of $\rho(g)$ are $1$,  $\lambda^2$ and $\lambda^{-2}$. Then $g \in G$ is a reflection if and only if exactly two of these eigenvalues are equal to 1, which is impossible in any characteristic.

On the other hand if 
\[h=\begin{pmatrix} \pm 1 & 1 \\ 0 & \pm 1 \end{pmatrix} \in \SL_2(F)\]
and 
\[A = \begin{pmatrix} a & b \\ c & -a \end{pmatrix}\]
then a quick calculation shows that $hA=Ah$ if and only if $c=0$ and $2a=0$. Thus we obtain 
\[\dim(\ker(\rho(g)-I)) = \left\{ \begin{array}{lr} 2 & q \ \text{even}\\ 1 & q \ \text{odd},  \end{array} \right.\]
for any $g \in G$ represented by a matrix whose Jordan form is $h$.  
Therefore there are no pseudoreflections in the odd case. In the even case, if $g$ has Jordan form $h$ as above then $\tr(g)=0$. Conversely, using \eqref{rho}, we see that the general element
\[g = \begin{bmatrix} a & b\\ c& a \end{bmatrix}\] of $G$ with trace zero fixes the hyperplane spanned by $v_2$ and $bv_1+cv_3$. 
\end{proof}


We are now ready to describe the algebra of invariants $R^G$. Translating the results of \cite{Maithani}  into our preferred basis, we define
\[f_1:= x_1x_3-x_2^2,\]
and
\[f_2:= x_1^qx_3+x_1x_3^q-2x_2^{q+1}.\]
Both are easily seen to be invariant under $G$. Indeed, $f_1$ is the determinant of a generic trace-free matrix $x_1v_1+x_2v_2+x_3v_3$ and $f_2 = \mathcal{P}^1(f_1)$, where $\mathcal{P}^1$ is the first Steenrod squaring operation in $\F_q$. We recommend \cite{SmithNeusel} as a resource for learning about applications of the Steenrod algebra to invariant theory. 

Fix a choice of irreducible quadratic \begin{equation}\label{g}g (x) = x^2 - \gamma x - \delta\end{equation} in $F$. Define
\[\tilde{f_3} = \prod_{\alpha \in F, \beta \in F^*} \left(x_1 - \frac{(2\alpha-\gamma)}{\beta}x_2 + \frac{g(\alpha)}{\beta^2} x_3 \right).\]

The following results follow from \cite[Section~7, 8]{Maithani} (note the change of sign and notation; the source uses $g(x)=x^2-\tau x+\delta$):
\begin{prop}\label{preinv}\
\begin{enumerate}
\item[(a)] There exists $f_3 \in R^G$ such that $f_3^2 = \tilde{f_3}$.

\item[(b)] For any appropriate choice of $f_3$, $\{f_1,f_2,f_3\}$ is a homogeneous system of parameters for $R^G$. \\

Let $\mathcal{A} = F[f_1,f_2,f_3] \subseteq R^G$. As $\{f_1,f_2,f_3\}$ form a homogeneous system of parameters, $\mathcal{A}$ is a polynomial algebra, and is independent of the choice of irreducible quadratic used to define $f_3$. Moreover we have:\\

\item[(c)] Suppose $q$ is even. Then $R^G = \mathcal{A}$.

\item[(d)] Suppose $q$ is odd. Then $R^G$ is a free module of rank 2 over $\mathcal{A}$, generated by $1$ and
\[f_4: =| \Jac(f_1,f_2,f_3)|.\]
\end{enumerate}
\end{prop}
In particular, this completes the proof of Theorem \ref{invars}.


\section{Covariants and Differential Invariants}

The goal of this section is to construct a minimal generating set of $\Omega^G$ for an arbitrary prime power $q$. Our strategy is to consider $(\Omega^i)^G$ separately for each $i$. We note that $\Lambda^0(V^*)$ is a trivial $FG$-module, and, since $\rho(G) \subseteq \SL_3(F)$, $\Lambda^3(V^*)$ is also trivial. Furthermore, $\Lambda^1(V^*) \cong V^*$ and $\Lambda^2(V^*) \cong V$. We see immediately that $(\Omega^0)^G$ and $(\Omega^3)^G$ are free $R^G$-modules generated by $1$ and $\dd x_1 \wedge \dd x_2 \wedge \dd x_3$ respectively, and that
\[(\Omega^1)^G \cong (R \otimes V^*)^G,\]
\[(\Omega^2)^G \cong (R \otimes V)^G.\]

We distinguish two cases:

\subsection{The even case}\label{sec:even}

If $q$ is even, then we have $R^G= \mathcal{A} = F[f_1,f_2,f_3]$. In particular, since $R^G$ has polynomial invariants, $G$ must act as a reflection group on $V$. By inspection, all the generators $\tau_e \ (e \in F^*)$ and $\sigma$ act as reflections on $V$. In fact, since $\rho(G) \subseteq \SL_3(F)$, all reflections in $G$ are transvections.

Recall from Lemma \ref{norefs} that the reflections in $G$ are precisely those elements with trace zero. The general element
\[g =\begin{bmatrix} a & b\\ c& a \end{bmatrix}\] of $G$ with trace zero fixes the hyperplane spanned by $v_2$ and $bv_1+cv_3$, in other words, it fixes $$U = \ker(l_{U}) = \ker(cx_1+bx_3).$$

Therefore, $V$ contains $q+1$ reflecting hyperplanes, and these are indexed by lines $(b:c) \in \mathbb{P}^1(F)$. The space of (transvection) root vectors of $U$ is
\[\left\{\frac{a^2-1}{\hat{c}}v_1+av_2+\hat{c}v_3: a \in F, (\hat{b},\hat{c}) \in (b:c)\right\}\] if $c \neq 0$ and
\[\left\{\hat{b}v_1+av_2+\frac{a^2-1}{\hat{b}}v_3: a \in F, (\hat{b},\hat{c}) \in (b:c)\right\}\] if $b \neq 0$ (these agree if $b$ and $c$ are both nonzero, since $a^2+bc=1$). Each space has dimension 2 unless $q=2$ in which case each has dimension 1.

In the notation of Section \ref{sec:refs}, we therefore have
\[Q = \prod_{(b:c) \in \mathbb{P}^1(F)}(cx_1+bx_3) = x_1^qx_3+x_1x_3^q = f_2,\]
and
\[\tilde{Q} =  \prod_{(b:c) \in \mathbb{P}^1(F)}(cx_1+bx_3)^2 = (x_1^qx_3+x_1x_3^q)^2 = f^2_2\]
if $q>2$, while $\tilde{Q} = f_2$ if $q=2$. Clearly we also have $Q_{\det}=1$.

Three obvious elements of $(\Omega^1)^G$ are $\dd f_1, \dd f_2$ and $\dd f_3.$ These have degrees $1$, $q$ and $\frac12 q(q-1)-1$  respectively. Note that
\[\dd f_1 = x_3 \dd x_1+ x_1 \dd x_3\]
and
\[\dd f_2 = x_3^q \dd x_1+ x_1^q \dd x_3.\]

Suppose $q=2$. Then $f_3 = x_1+x_2+x_3$ and 
\begin{equation}\label{df3q2}\dd f_3 = \dd x_1+\dd x_2+ \dd x_3.\end{equation}
By direct calculation we have
\[\dd f_1 \wedge \dd f_2 \wedge \dd f_3 = f_2 \dd x_1 \wedge \dd x_2 \wedge \dd x_3.\] Then Proposition \ref{1formcriterion} implies immediately:

\begin{prop}
Let $q=2$. Then $(\Omega^1)^G$ is a free $R^G$-module generated by $\dd f_1, \dd f_2, \dd f_3$.
\end{prop}

Three obvious elements of $(\Omega^2)^G$ are $\dd f_1 \wedge \dd f_2, \dd f_2 \wedge \dd f_3, \dd f_3 \wedge \dd f_1$. In case $q=2$, their degrees are $3$, $2$ and $1$ respectively. The sum of their degrees is 6 while the degree of $Q=f_2$ is 3. Observe that $\dd f_1 \wedge \dd f_2$ is divisible by $f_2$, and
\begin{equation}\label{df1wedgedf2} \frac{\dd f_1 \wedge \dd f_2}{f_2} = \dd x_1 \wedge \dd x_3 \in (\Omega^2)^G \end{equation}  has polynomial degree zero. Further,
$$ \dd f_2 \wedge \dd f_3 = x_3 \dd x_1 \wedge \dd x_2 + x_3 \dd x_2 \wedge \dd x_3 + (x_1+x_3)\dd x_3 \wedge \dd x_1,$$ 
and
$$ \dd f_3 \wedge \dd f_1 = x^2_3 \dd x_1 \wedge \dd x_2 + x^2_3 \dd x_2 \wedge \dd x_3 + (x^2_1+x^2_3)\dd x_3 \wedge \dd x_1.$$ 
These have degree sum $0+1+2 = 3 = \deg(f_2)$ and are clearly linearly independent over $\Quot(R)$ as their degrees are distinct. Now Proposition \ref{derivcriterion} implies immediately:

\begin{prop}
Let $q=2$. Then $(\Omega^2)^G$ is a free $R^G$-module generated by $f_2^{-1} \dd f_1 \wedge \dd f_2, \dd f_2 \wedge \dd f_3, \dd f_3 \wedge \dd f_1$.
\end{prop}

This completes the proof of Theorem \ref{even} in the case $q=2$. To prove Theorem \ref{evenmin}, note that the $R^G$-module generating set is an algebra generating set. It is clear that $\dd f_1 \wedge \dd f_3$ and $\dd f_2 \wedge \dd f_3$ can be omitted, and further 
\begin{align*} (f_2^{-1}\dd f_1 \wedge \dd f_2) \wedge \dd f_3  &= (\dd x_1 \wedge \dd x_3) \wedge (\dd x_1+\dd x_2 + \dd x_3) \\
&= \dd x_1 \wedge \dd x_2 \wedge \dd x_3 \end{align*}
by \eqref{df3q2} and \eqref{df1wedgedf2}, so the latter can also be omitted. Minimality now follows from freeness over $R^G$ and from degree considerations.

Now suppose $q>2$. The degree sum of $\dd f_1, \dd f_2$ and $\dd f_3$ is then $\frac12 q(q+1)$ and this exceeds $\deg(\tilde{Q}) = \deg(f_2^2) = 2(q+1)$ if $q>4$. Instead we consider
\[\omega:= (x_2^qx_3+x_2x_3^q)\dd x_1+  (x_1^qx_3+x_1x_3^q)\dd x_2 + (x_1^qx_2+x_1x_2^q)\dd x_3,\]
with degree $q+1$. (If $q=4$, then $\omega^2 = \tilde{f_3}$ so we can take $\omega=f_3$.) Now easy direct calculations show that $\omega \in (\Omega^1)^G$ and furthermore
\[\dd f_1 \wedge \dd f_2 \wedge \omega = f_2^2.\]

Then Proposition \ref{1formcriterion} implies immediately:
\begin{prop}
Let $q>2$. Then $(\Omega^1)^G$ is a free $R^G$-module generated by $\dd f_1, \dd f_2$ and $\omega$.
\end{prop}
 
Now when $q>2$ we have $b_U=2$ for all $U \in \mathcal{U}$, in other words, all transvection root spaces are maximal. Therefore \cite[Theorem~4.6]{HansonShepler} applies: the wedge product of each pair in $\{\dd f_1, \dd f_2, \omega\}$ is divisible by $f_2$, and the three 2-forms $f_2^{-1} \dd f_1 \wedge \dd f_2,f_2^{-1} \dd f_1 \wedge \omega$ and $f_2^{-1} \dd f_2 \wedge \omega$ generate $(\Omega^2)^G$ as a free module. We note here that

\begin{align*} \frac{\dd f_1 \wedge \dd f_2}{f_2} &= \dd x_1 \wedge \dd x_3\\ \frac{{\dd f_1 \wedge \omega}}{f_2} &= \sum_{i=1}^3 x_i \dd x_{i+1} \wedge \dd x_{i+2}, \\  \frac{{\dd f_2 \wedge \omega}}{f_2} &= \sum_{i=1}^3 x^q_i \dd x_{i+1} \wedge \dd x_{i+2}. \end{align*}

This completes the proof of Theorem \ref{even} in the case $q>2$. Theorem \ref{evenmin} follows easily from freeness over $R^G$ and degree considerations: wedging distinct 1-forms of polynomial degrees in the set $\{1, q, q+1\}$ can never produce a form of polynomial degree $<q+1$. 

More concretely, the conclusion of Theorem 4.6 in loc. cit. is that $\Omega^G$ is a free twisted exterior $R^G$ algebra generated by $\dd f_1, \dd f_2, \omega$ under the twisted wedging operator $\curlywedge$ defined by
\[(-)\curlywedge (-) = \frac{(-) \wedge (-)}{f_2}.\]

\subsection{The odd case}\label{sec:odd}

In this section we assume $q$ is odd. Since a minimal generating set for $\Omega^G$ when $q=3$ was constructed in \cite{ElmerMeyer}, we also assume that $q>3$ throughout. Recall that if $q$ is odd then $R^G$ is a free module of rank 2 over $\mathcal{A}$. Our strategy is first to construct a $\mathcal{A}$-generating set of each summand $(\Omega^i)^G$ for $i=0,1,2,3$. We begin by showing:

\begin{Lemma}
$(\Omega^i)^G$ is a Cohen-Macaulay $R^G$ module for each $i$.
\end{Lemma}

\begin{proof} We may write $\Omega^G = \bigoplus_{i=0}^3 (R \otimes W_i)^G$ where $W_i = \Lambda^i(V^*)$. Note that $\dim(V^P) =1$. By \cite[Proposition~6(ii)]{BroerChuaiRelative}, $(R \otimes W_i)^P$ is a Cohen-Macaulay $R^P$-module for each $i$, then by Proposition 5 (loc. cit.)  
$(R \otimes W_i)^G$ is a Cohen-Macaulay $R^G$-module for each $i$.
\end{proof}

It now follows that $(\Omega^i)^G$ is a free $\mathcal{A}$-module for each $i$. 

Free and Cohen-Macaulay modules of covariants were studied by the author in \cite{ElmerCMCov}, from where we obtain the following result, which will be a key tool in all that follows:

\begin{prop}\label{freetest} Let $A$ be a graded polynomial ring over $\kk$ with $A_0= \kk$ and let $M$ be a finitely generated free graded module over $A$. Let $m_1, m_2, \ldots, m_r$ be a $A$-independent set of elements of $M$, where $r = r(M,A)$. Then
\[\sum_{i=1}^r \deg(g_i) \geq s(M,A)\] with equality if and only if $M$ is generated by $m_1, \ldots, m_r$.
\end{prop}

In the above, $r(M,A)$ denotes the rank of $M$ as an $A$-module and $s(M,A)$ is the so-called {\it $s$-invariant}; this is defined to be the coefficient of $(t-1)$ in the expansion of $\frac{H(M,t)}{H(A,t)}$ about $t=1$.

We will need an explicit choice of $f_3$. We define 
 \begin{equation}\label{L} L:= \{(b,c): b \in F, c \in F^*, b^2-4c \ \text{not square}.\}.\end{equation}
As $q$ is odd, we may assume $\gamma = 0$ in the definition \eqref{g} of $g(x)$, and choose any $\delta \in F$ nonsquare. Now
consider the mapping $\Phi: F \times F^* \rightarrow F \times F^*$ defined by
\[\Phi(\alpha, \beta) = \left(\frac{-2 \alpha}{\beta}, \frac{\alpha^2-\delta}{\beta^2} \right).\]
We claim that:

\begin{Lemma}
The image of $\Phi$ is $L$. Moreover, every nonempty fibre of $\Phi$ has size two.
\end{Lemma}

\begin{proof}
Let $\alpha \in F, \beta \in F^*$ and set $\Phi(\alpha, \beta) = (b,c)$. Then
\[b^2-4c = \frac{4 \delta}{\beta^2} = \left(\frac{2}{\beta}\right)^2 \delta\] is not a square in $F$, so $(b,c) \in L$.

Conversely suppose $(b,c) \in L$ and $(\alpha, \beta) \in \Phi^{-1}(b,c)$. Write $b^2-4c = \delta \epsilon^2$ for some $\epsilon \in F$, which is unique up to sign. Then
\[\delta \epsilon^2 = b^2-4c =  \left(\frac{2}{\beta}\right)^2 \delta\] which shows that
\[\beta = \pm 2 \epsilon^{-1},\]
and \[\alpha =\frac{- b \beta}{2} = \mp b \epsilon^{-1}.\]
\end{proof}

Now define \begin{equation}\label{f3} f_3 = \prod_{(b,c) \in L} (x_1+bx_2+cx_3).\end{equation} The Lemma above implies that $f_3^2 = \tilde{f_3}$. Moreover, for any $g \in G$ we have
\[f_3^2 = g \cdot f_3^2 = (g \cdot f_3)^2\] so that $g \cdot f_3 = \pm f_3$. But since there is no nontrivial character $\PSL_2(F) \rightarrow F^*$ we must have $g \cdot f_3= f_3$ for any $g \in G$ and hence $f_3 \in R^G$. We deduce that $f_3$ defined in this way has the properties listed in Proposition \ref{preinv}.

\begin{rem}
Smith \cite{Smith2x2} takes \eqref{f3} as the definition of $f_3$ and proves directly that $\{f_1,f_2,f_3\}$ is a homogeneous system of parameters for $R^G$. However, he also claims that that $f_3 \in R^{GL_2(F)}$ and concludes that $F[f_1,f_2,f_3] = R^{GL(F)}$. This is not true - in fact $R^{GL_2(F)}$ is a hypersurface and a free module of dimension 2 over $F[f_1,f_2,f_3^2]$ as observed in \cite{Maithani}.
\end{rem} 

We return to computing covariants. Again, $W_0$ and $W_3$ are trivial modules, so it is immediate that $$(\Omega^0)^G = \mathcal{A}(1,f_4)$$ and $$(\Omega^3)^G = \mathcal{A}( \dd x_1 \wedge \dd x_2 \wedge \dd x_3,f_4 \dd x_1 \wedge \dd x_2 \wedge \dd x_3).$$ It remains to find free generators for $(\Omega^i)^G = (R \otimes W_i)^G$ over $\mathcal{A}$ for $i=1,2$. 

As $\dim(V)=3$ we have $W_1 = V^*$ and $W_2 = \Lambda^2(V^*) \cong V$. Notice that $V \cong S^2(M)$ where $M$ is the natural 2-dimensional $\SL_2(F)$ module. By, for instance \cite{WildonMcDowell}, since $p>2$ we have $V \cong V^*$, so the modules $(\Omega^1)^G$ and $(\Omega^2)^G$ are isomorphic. We will need to know an explicit isomorphism. Define a map $\phi: W_2 \rightarrow W_1$ as follows:

\begin{align*}
\phi(\dd x_2 \wedge \dd x_3) &=  \dd x_3,\\
\phi(\dd x_3 \wedge \dd x_1) &=  -2\dd x_2,\\
\phi(\dd x_1 \wedge \dd x_2) &= \dd x_1.\\
\end{align*}

\begin{Lemma}\label{isom}
$\phi$ is an isomorphism.
\end{Lemma}

\begin{proof}
This is easily verified.
\end{proof}


It follows that $\phi$ induces an isomorphism of $R^G$-modules $(\Omega^2)^G \rightarrow (\Omega^1)^G$. We will abuse notation slightly by denoting the induced map by $\phi$.

\begin{prop}
We have, for $i=1,2$,
\[r((\Omega^i)^G,\mathcal{A}) = 6\]
and
\[s((\Omega^i)^G,\mathcal{A}) = \frac32 q(q+1).\]
\end{prop}

\begin{proof}
By \cite[Lemma~1]{BroerChuaiRelative} and Proposition \ref{preinv}(d) we have
$$r((\Omega^i)^G,\mathcal{A}) = r((\Omega^i)^G,R^G)r(R^G,\mathcal{A}) = 3 \times 2 = 6.$$

Note that $$\deg(f_4) = \deg(f_1)+\deg(f_2)+\deg(f_3) - 3 = 2+(q+1)+\frac12q(q-1)-3= \frac12 q(q+1).$$ Therefore by \cite[Lemma~1, Theorem~1]{BroerChuaiRelative}, we have
\begin{align*}
s((\Omega^i)^G,\mathcal{A}) &= r((\Omega^i)^G,R^G)s(R^G,\mathcal{A})+s((\Omega^i)^G,R^G)r(R^G,\mathcal{A})\\
&= 3 \times \frac12 q(q+1) + 0 \times 2\\
&=  \frac32 q(q+1)
\end{align*}
with the zero on the second line coming from the fact that $G$ contains no reflections (Lemma \ref{norefs}).
\end{proof}

It follows that for both $i=1$ and $i=2$, it is enough to find a set of 6 elements of $\Omega^i(M)^G$ which are $\mathcal{A}$-independent and whose degrees sum to $\frac32 q(q+1).$

Three obvious elements of $(\Omega^1)^G$ are given by $b_1:=\dd f_1, b_2:=\dd f_2,b_3:=\dd f_3$ with degrees 1, $q$ and $\frac12q(q-1)-1$ respectively.  Notice that the coefficients of $b_i$ relative to the basis $\dd x_1,\dd x_2,\dd x_3$ of $(\Omega^1)$ are given by the rows of the Jacobian matrix $J = \Jac(f_1,f_2,f_3)$.

Three obvious elements of $(\Omega^2)^G$ are given by $c_1 = \dd f_2 \wedge \dd f_3, c_2= \dd f_3 \wedge \dd f_1, c_3= \dd f_1 \wedge \dd f_2$ with degrees $\frac12q (q+1)-1$, $\frac12 q(q-1)$ and $q+1$ respectively. Notice that the coefficients of $c_i$ relative to the  basis $\dd x_2 \wedge \dd x_3, \dd x_3 \wedge \dd x_1, \dd x_1 \wedge \dd x_2$ of $\Omega^2$ are given by the rows of the matrix of cofactors of $J$.

For each $i=1,2,3$ define
\[b_{i+3} = \phi(c_i)\]
and
\[c_{i+3} = \phi^{-1}(b_i).\]

We claim that $b_1,b_2, \ldots, b_6$ generate $(\Omega^1)^G$ over $\mathcal{A}$ and $c_1,c_2, \ldots, c_6$ generate $(\Omega^2)^G$ over $\mathcal{A}$. Notice that the degree sum of either is
\begin{align*}& 1+q+\frac12 q(q-1)-1+\frac12q (q+1)-1+\frac12 q(q-1)+q+1\\
&=q+ q(q-1)+\frac12q (q+1)+q\\
&=\frac32 q(q+1), \end{align*}
so it is enough to show these elements are $\mathcal{A}$-independent.

In order to do this, we will need to know the lead terms of $\frac{\partial f_3}{\partial x_i}$ for $i=1,2,3$.
\begin{Lemma}\label{leadterms}
 The lead term of  $\frac{\partial f_3}{\partial x_i}$ is:
\[\left\{ \begin{array}{lr} -\frac{q-1}{2}x_1^{\frac12 q(q-3)}x_2^{q-1} & i=1; \\
 \frac{q-1}{2}x_1^{\frac12 q(q-3)+1}x_2^{q-2} & i=2;\\
 -\frac{q-1}{8}x_1^{\frac12 q(q-3)+2}x_2^{q-3} & i=3. \end{array}\right.\]
\end{Lemma}

\begin{proof}
Suppose first that $i=1$. Then $\frac{\partial f_3}{\partial x_1}$ is divisible by $x_3$ if and only if $f_3$ is so, if and only if $f_3(x_1,x_2,0) = 0$. We have
\[f_3(x_1,x_2,0) = \prod_{(b,c) \in L} (x_1+bx_2) = \prod_{b \in F} (x_1+bx_2)^{\frac12 (q-1)}\]
since for each $b \in F$ there are $\frac12(q-1)$ values of $c \in F^{*}$ with $b^2-4c$ not square. This can be written as
\[(x_1^{q}-x_1x_2^{q-1})^{\frac12(q-1)}\] which is clearly nonzero. This shows that $f_3$ is not divisible by $x_3$, and hence neither is $\frac{\partial f_3}{\partial x_1}$. Further, the lead term of $\frac{\partial f_3}{\partial x_1}$ is the lead term of
$\frac{\partial}{\partial x_1} f_3(x_1,x_2,0)$, since specialisation to $x_3=0$ and differentiation with respect to $x_1$ commute.
Now we have
\begin{align*}  \frac{\partial}{\partial x_1} (x_1^{q}-x_1x_2^{q-1})^{\frac12(q-1)}
=& -\frac12(q-1)x_2^{q-1}(x_1^q-x_1x_2^{q-1})^{\frac12(q-3)}\\
=& -\frac12(q-1) x_2^{q-1}x_1^{\frac12 q(q-3)}+\ \text{smaller terms}
\end{align*}
as required.

 The proof for $i=2$ is similar and left to the reader. For $i=3$ the proof is quite different, since differentiation with respect to $x_3$ does not commute with evaluation at $x_3=0$.

Viewed as a polynomial in $x_3$ alone, the coefficient of $x_3$ in $f_3$ is
\[\sum_{(b,c) \in L} c \prod_{(b',c') \in L \setminus \{(b,c)\} } (x_1+b'x_2).\]
Provided this is not zero, it is the lead term of $\frac{\partial f_3}{\partial x_3}$. We can rewrite it as
\[f_3(x_1,x_2,0) \sum_{(b,c) \in L}\frac{c}{x_1+bx_2}.\]
Let
\[S(x_1,x_2) =  \sum_{(b,c) \in L}\frac{c}{x_1+bx_2} = \sum_{b \in F} \frac{A_b}{x_1+bx_2}\] where 
\[A_b = \sum_{c \in F^*, b^2-4c \in Z} c\] where $Z$ is the set of non-squares in $F$. We can write this as
\[\sum_{d \in Z} \frac14 (b^2-d).\]
Now since $q>3$ it is well known that $\sum_{d \in Z} d = 0$, see for instance \cite[Section~5.1]{LidlNiederreiter}.

Hence
\[A_b =\frac{q-1}{8} b^2\]
and therefore
\[S(x_1,x_2) = \frac{q-1}{8} \sum_{b \in F} \frac{b^2}{x_1+bx_2}.\]
Now we have
\begin{align*}
\sum_{b \in F} \frac{b^2}{x_1+bx_2} &= \frac{(b-\frac{x_1}{x_2})(b+\frac{x_1}{x_2})+ \frac{x_1^2}{x_2^2}}{x_1+bx_2}\\
&= \frac{1}{x_2} \sum_{b \in F} (b-\frac{x_1}{x_2}) + \frac{x_1^2}{x_2^3} \sum_{b \in F} \frac{1}{b+ \frac{x_1}{x_2}}\\
&= \frac{x_1^2}{x_2^3} \sum_{b \in F} \frac{1}{b+ \frac{x_1}{x_2}}  \end{align*}
To evaluate this, recall that 
\[F(z):= \prod_{b \in F} (z+b) = z^q-z.\]
Differentiating both sides we get
\[F'(z) = \sum_{b \in F} \sum_{b' \in F, b' \not b} (z+b) = -1.\]
Hence
\[\sum_{b \in F} \frac{1}{z+b} = \frac{F'(z)}{F(z)} = -\frac{1}{z^q-z}.\]
From this we deduce that
\[S(x_1,x_2) = -\frac{q-1}{8}\frac{x_1^2}{x_2^3} \frac{1}{\frac{x_1^q}{x_2^q}-\frac{x_1}{x_2}} =  - \frac{q-1}{8}x_1^2x_2^{q-3} \frac{1}{x_1^q-x_1x_2^q}.\]
Hence
\[f_3(x_1,x_2,0)S(x_1,x_2) = -\frac{q-1}{8} x_1^2 x_2^{q-3} (x_1^q-x_1x_2^{q-1})^{\frac{p-3}{2}}\]
with lead term
\[-\frac{q-1}{8} x_1^{q(q-3)+2}x_2^{q-3}\]
as desired.
\end{proof}

We can use this result to find the lead terms of the coefficients of $c_1,c_2, c_3$, which we display in the table below:

\vspace{1cm}
\begin{tabular}{|c||c|c|c|} \hline
 & $\dd x_2 \wedge \dd x_3$ &  $\dd x_3 \wedge \dd x_1$ &  $\dd x_1 \wedge \dd x_2$ \\ \hline \hline
$c_1$ & $-\frac{q-1}{2}x_1^{\frac{1}{2}q(q-1)+1}x_2^{q-2}$ & $ -\frac{q-1}{2}x_1^{\frac{1}{2}q(q-1)}x_2^{q-1}$ &  $x_1^{\frac{1}{2}q(q-3)}x_2^{2q-1}$\\ \hline
$c_2$ & $\frac{q-1}{4} x_1^{\frac12 q(q-3)+2}x_2^{q-2}$ & $\frac{q-1}{2} x_1^{\frac12 q(q-3)+1}x_2^{q-1}$ &  $x_1^{\frac12 q(q-3)}x_2^{q}$\\  \hline
 $c_3$ & $ -2x_1^qx_2$ & $-x_1^qx_3$ & $-2x_2^qx_3$ \\ \hline
\end{tabular}
\vspace{1cm}

(the entries on the the $c_3$ row are the exact coefficients.) Now we can use this table and $\phi$ to find the lead terms of the coefficients of $b_4,b_5,b_6$; we include them in the table below along with the lead terms of the coefficients of $b_1,b_2,b_3$.

\vspace{1cm}
\begin{tabular}{|c||c|c|c|} \hline
 & $\dd x_1$ &  $\dd x_2$ &  $\dd x_3$ \\ \hline \hline
$b_1$ & $x_3$ & $-2x_2$ & $x_1$\\ \hline
$b_2$ & $x_3^q$ & $-2x_2^q$ & $x_1^q$\\ \hline
$b_3$ & $-\frac{q-1}{2}x_1^{\frac12 q(q-3)}x_2^{q-1}$ & $\frac{q-1}{2}x_1^{\frac12 q(q-3)+1}x_2^{q-2}$ & $-\frac{q-1}{8}x_1^{\frac12 q(q-3)+2}x_2^{q-3}$\\ \hline
$b_4$ &  $x_1^{\frac{1}{2}q(q-3)}x_2^{2q-1}$ & $ -x_1^{\frac{1}{2}q(q-1)}x_2^{q-1}$   & $-\frac{q-1}{2}x_1^{\frac{1}{2}q(q-1)+1}x_2^{q-2}$\\ \hline
$b_5$ &  $x_1^{\frac12 q(q-3)}x_2^{q}$& $ x_1^{\frac12 q(q-3)+1}x_2^{q-1}$   & $\frac{q-1}{4} x_1^{\frac12 q(q-3)+2}x_2^{q-2}$\\ \hline
 $b_6$ & $-2x_2^qx_3$  & $2x_1^qx_3$ & $ -2x_1^qx_2$ \\ \hline
\end{tabular}
\vspace{1cm}

\begin{prop} The 1-forms $b_1,b_2,b_3, \ldots, b_6$ are $\mathcal{A}$-indepedent and generate $(\Omega^1)^G$.
\end{prop}

\begin{proof}
Suppose $\sum_{i=1}^6 a_ib_i=0$. If this relation is to hold and not all $a_i$ are zero, there must be a pair $i,j$ such that
\[\LT(a_i)\LT(b_{3,i})= \LT(a_j)\LT(b_{3,j})\] where $b_{3,i}$ is the coefficient of $\dd x_3$ in $b_i$. Now recall that every element of $\mathcal{A}$ has lead term $x_1^{Nk}x_2^{2l}$ where $N = \frac12q(q-1)$. Since $q \geq 5$ we have $N>2q-2$. Now we see by inspection that the $(x_1,x_2)$-bidegrees of the lead terms of $b_{3,i}$ for $i=1, \ldots, 6$ are distinct modulo $(N,2)$. Therefore the set is $\mathcal{A}$-independent as required. The generation follows on applying Proposition \ref{freetest} because $\sum_{i=1}^6 \deg(b_i) = \frac32 q(q+1) = s(\Omega^1(M)^G, \mathcal{A})$.
\end{proof}

\begin{cor}
The 2-forms $c_1,c_2, \ldots, c_6$ are $\mathcal{A}$-independent and generate $(\Omega^2)^G$.
\end{cor}

\begin{proof}
This is immediate, since $\phi^{-1}: (\Omega^1)^G \rightarrow (\Omega^2)^G$ is a $R^G$-module isomorphism, hence a $\mathcal{A}$-module isomorphism, and $c_i = \phi^{-1}(b_i)$ for all $i=1, \ldots, 6$.
\end{proof}

This completes the proof of Theorem \ref{odd} in the case $q \neq 3$.

\begin{proof}[Proof of Corollary \ref{oddmin}]
It is clear that the set $\{f_1,f_2,f_3\} \cup S$ generates $\Omega^G$ and that $c_1,c_2,c_3$ and $f_4 \dd x_1 \wedge \dd x_2 \wedge \dd x_3$ can be omitted. It is clear that $\dd x_1 \wedge \dd x_2 \wedge \dd x_3$ cannot be omitted. Since $b_1, \ldots, b_6$ generate freely over $F[f_1,f_2,f_3]$ and each has $x_1$-degree strictly smaller than the $x_1$-degree of $f_4$, none of these can be ommited. It remains to show that none of $c_4,c_5, c_6$ can be omitted.

The lead terms of coefficients of these $2$-forms are given in the table below:

\vspace{1cm}
\begin{tabular}{|c||c|c|c|} \hline
 & $\dd x_2 \wedge \dd x_3$ &  $\dd x_3 \wedge \dd x_1$ &  $\dd x_1 \wedge \dd x_2$ \\ \hline \hline
$c_4$  & $x_1$ & $x_2$  & $x_3$\\ \hline
$c_5$  &  $x_1^q$ &  $x_2^q$ & $x_3^q$\\ \hline
$c_6$  &  $-\frac{q-1}{8}x_1^{\frac12 q(q-3)+2}x_2^{q-3}$ & $-\frac{q-1}{4}x_1^{\frac12 q(q-3)+1}x_2^{q-2}$  & $-\frac{q-1}{2}x_1^{\frac12 q(q-3)}x_2^{q-1}$\\ \hline
\end{tabular}
\vspace{1cm}

Again each has $x_1$-degree strictly smaller than that of $f_4$, so we see that $\{c_1,c_2, \ldots, c_6\}$ is a minimal generating set for $(\Omega^2)^G$ over $R^G$. It remains to show that none of $c_4,c_5,c_6$ lie in the $\mathcal{A}$-submodule of $\Omega^2$ generated by wedge products of $b_4,b_5, b_6$. But any such wedge product has $x_1$-degree at least $\frac12 q(q-1)$ which is greater than the $x_1$-degree of any of these.
 \end{proof}

\subsection{The missing case: q=3}\label{sec:three}

Our proof of Theorem \ref{odd} is invalid if $q=3$, since it relies on the fact that the sum of the nonsquares in $F$ is zero. However, a minimal generating set for $\Omega^G$ when $q=3$ is given in \cite{ElmerMeyer}. In this section we will verify that Theorem \ref{odd} and Corollary \ref{oddmin} also hold when $q=3$.

In \cite{ElmerMeyer} we worked with respect to a different basis of $V$:
\[v_1 = \begin{pmatrix} 0 & 1 \\ -1 & 0 \end{pmatrix}, v_2 =  \begin{pmatrix} -1 & -1 \\ -1 & 1 \end{pmatrix}, v_3 =  \begin{pmatrix} 1 & -1 \\ -1 & -1 \end{pmatrix} \] with $x_1,x_2, x_3$ the dual basis as usual. This choice was made in order to exploit the fact that in this case $V$ is induced from a non-modular representation of a subgroup of $G$. An additional benefit is the observation that in this basis, $\rho(g)^{-T}= \rho(g)$ for all $g \in G$; consequently the actions of $G$ on $V$ and $V^*$ are given by the same matrices, and the map $\phi$ from the present article  with respect to this basis is simply 
\begin{align*}
\phi(\dd x_2 \wedge \dd x_3) &=  \dd x_1;\\
\phi(\dd x_3 \wedge \dd x_1) &=  \dd x_2;\\
\phi(\dd x_1 \wedge \dd x_2) &= \dd x_3.
\end{align*}

Generators of a subalgebra $\mathcal{A} \subseteq R^G$ with respect to this basis were given in \cite[Proposition~4]{ElmerMeyer} as
\begin{align*} a_1 &= x_1^2+x_2^2+x_3^2;\\
a_2 &= x_1x_2x_3;\\
a_3 &= x_1^4+x_2^4+x_3^4.
\end{align*}
It is easy to see that $a_1$ is the determinant of generic trace-free matrix, and that $a_3=-\mathcal{P}^1(a_1)$. Moreover Theorem \ref{invars} implies that $f_2$ is the unique invariant of degree 3 up to a constant, so this must be proportional to $a_2$ in our new basis. The nontrivial secondary generator of $R^G$ given in \cite{ElmerMeyer} is
\[b: = x_1^4x_2^2+x_1^2x_3^4+x_2^4x_3^2.\]
This is not $\Jac(a_1,a_2,a_3)$ but rather $a_1^3-a_1a_3-\Jac(a_1,a_2,a_3)$, so we can certainly replace it by $\Jac(a_1,a_2,a_3)$. Therefore the generating sets given for $\Omega^0$ and $\Omega^3$ are equivalent in the two sources.

Corollary 8 in \cite{ElmerMeyer} states, after unifying our notation:
\begin{itemize}
\item[(a)] $(\Omega^1)^G$ is generated as an $\mathcal{A}$-module by the set
\begin{align*} \{c_1 &= x_1 \dd x_1+x_2\dd x_2+x_3\dd x_3,\\
c_2 &= x_2x_3\dd x_1+x_3x_1\dd x_2+x_1x_2\dd x_3,\\
c_3 &= x_1^3\dd x_1+x_2^3\dd x_2+x_3^3\dd x_3,\\
c_4 &=x_1x_2^2\dd x_1+x_2x_3^2\dd x_2+x_3x_1^2\dd x_3,\\
c_5 &= x_2^3x_3\dd x_1+x_3^3x_1\dd x_2+x_1^3x_2\dd x_3,\\
c_6 &= x_1^3x_2^2\dd x_1+x_2^3x_3^2\dd x_2+x_3^3x_1^2\dd x_3\}.\end{align*}
\item[(b)]  $(\Omega^2)^G$ is generated as an $A$-module by the set
\begin{align*} \{d_1 &= x_1\dd x_2\dd x_3+x_2\dd x_3\dd x_1+x_3\dd x_1\dd x_2,\\ 
d_2 &= x_2x_3\dd x_2\dd x_3+x_3x_1\dd x_3\dd x_1+x_1x_2\dd x_1\dd x_2,\\
d_3 &= x_1^3\dd x_2\dd x_3+x_2^3\dd x_3\dd x_1+x_3^3\dd x_1\dd x_2,\\
d_4 &=x_1x_2^2\dd x_2\dd x_3+x_2x_3^2\dd x_3\dd x_1+x_3x_1^2\dd x_1\dd x_2,\\
d_5 &= x_2^3x_3\dd x_2\dd x_3+x_3^3x_1\dd x_3\dd x_1+x_1^3x_2\dd x_1\dd x_2,\\
d_6 &= x_1^3x_2^2\dd x_2\dd x_3+x_2^3x_3^2\dd x_3\dd x_1+x_3^3x_1^2\dd x_1\dd x_2\}.\end{align*}
\end{itemize}

We see easily that, up to sign, $c_i = \dd a_i$ for $i=1,2,3$, and that $\phi(d_i)=c_i$ for each $i$. Moreover, the relations given in the proof of Theorem 9 in \cite{ElmerMeyer} confirm that $d_4,d_5,d_6$ may be replaced by $b_1 \wedge b_2, b_3 \wedge b_1$ and $b_2 \wedge b_3$ respectively, and so $c_4,c_5,c_6$ may be replaced by $\phi^{-1}(b_1 \wedge b_2), \phi^{-1}(b_3 \wedge b_1)$ and $\phi^{-1}(b_2 \wedge b_3)$. This completes the proof of Theorem \ref{odd} in the case $q=3$. Corollary \ref{oddmin} now follows with exactly the same proof.

\bibliographystyle{plain}
\bibliography{MyBib}

\end{document}